\documentclass[12pt]{article}
\usepackage{fullpage}
\usepackage{style}

\newcommand\de{\delta}
\newcommand\De{\Delta}

\newcommand\la{\lambda}

\renewcommand\th{\theta} %--- Latex uses \th for a Norse character
\newcommand\Th{\Theta}

\newcommand\re{{\mathbb R}}%reals

\DeclareMathOperator{\im}{im}

\def\x{\mathbf{x}}
\def\1{\mathbf{1}}
\DeclareMathOperator\diag{diag}
\DeclareMathOperator\dist{dist}
\DeclareMathOperator{\tr}{trs}
\DeclareMathOperator{\coker}{coker}
\DeclareMathOperator\IN{in}
\DeclareMathOperator\minors{minors}
\DeclareMathOperator\SNF{SNF}
\DeclareMathOperator\SP{sp}

\begin{document}
\title{The Degree-Distance and Transmission-Adjacency Matrices}
\author{Carlos A. Alfaro\thanks{
{\tt alfaromontufar@gmail.com}, 
Banco de M\'exico, 
Mexico City, Mexico} \and
Octavio Zapata\thanks{
{\tt octavioz@ciencias.unam.mx},
Departamento de Matemáticas, Facultad de Ciencias,
Universidad Nacional Autónoma de México, 
Mexico City, Mexico}}
\date{}
\maketitle
\begin{abstract}
    Let $G$ be a connected graph with adjacency matrix $A(G)$. 
    The distance matrix $D(G)$ of $G$ has rows and columns indexed by $V(G)$ with $uv$-entry equal to the distance $\dist(u,v)$ which is the number of edges in a shortest path between the vertices $u$ and $v$.
    The transmission $\tr(u)$ of $u$ is defined as $\sum_{v\in V(G)}\dist(u,v)$.
    Let $\tr(G)$ be the diagonal matrix with the transmissions of the vertices of $G$ in the diagonal, and $\deg(G)$ the diagonal matrix with the degrees of the vertices in the diagonal.
    In this paper we investigate the Smith normal form (SNF) and the spectrum of the matrices $D^{\deg}_+(G):=\deg(G)+D(G)$, $D^{\deg}(G):=\deg(G)-D(G)$, $A^{\tr}_+(G):=\tr(G)+A(G)$ and $A^{\tr}(G):=\tr(G)-A(G)$.
    In particular, we explore how good the spectrum and the SNF of these matrices are for determining graphs up to isomorphism.
    We found that the SNF of $A^{\tr}$ has an interesting behaviour when compared with other classical matrices. 
    We note that the SNF of $A^{\tr}$ can be used to compute the structure of the sandpile group of certain graphs.
    We compute the SNF of $D^{\deg}_+$, $D^{\deg}$, $A^{\tr}_+$ and $A^{\tr}$ for several graph families.
    We prove that complete graphs are determined by the SNF of $D^{\deg}_+$, $D^{\deg}$, $A^{\tr}_+$ and $A^{\tr}$.
    Finally, we derive some  results about the spectrum of $D^{\deg}$ and $A^{\tr}$.
\end{abstract}

\section{Introduction}

We will focus on matrices derived from the adjacency and the distance matrices of graphs, for this we need the following definitions.
Let $G$ be a simple connected graph with $n$ vertices.
Let $A(G)$ denote the adjacency matrix of $G$ and let $\deg(G)$ denote the diagonal matrix with the degrees of the vertices of $G$ in the diagonal.
The \emph{distance} $\dist(u,v)$ between two vertices $u$ and $v$ is the number of edges in a shortest path between them.
The \emph{distance matrix} $D(G)$ of $G$ is the matrix with rows and columns indexed by the vertices of $G$ with the $uv$-entry equal to $\dist(u,v)$.
Let $\tr(u)$ denote \emph{transmission} of vertex $u$, which is defined as $\sum_{v\in V(G)}\dist(u,v)$.
Let $\tr(G)$ denote the diagonal matrix with the transmissions of the vertices in the diagonal.
We define 
\begin{itemize}
    \item the \emph{degree-distance} matrix $D^{\deg}(G)$ of $G$ as $\deg(G)-D(G)$,
    \item the \emph{transmission-adjacency} matrix $A^{\tr}(G)$ of $G$ as $\tr(G)-A(G)$,
    \item the \emph{signless degree-distance} matrix $D^{\deg}_+(G)$ of $G$ as $\deg(G)+D(G)$, and 
    \item the \emph{signless transmission-adjacency} matrix $A^{\tr}_+(G)$ of $G$ as $\tr(G)+A(G)$.
\end{itemize}

One of the objectives of spectral graph theory is to understand to what extent graphs are characterized by their spectra.
Starting from the eigenvalues of a matrix $M$ associated with a graph, it seeks to deduce combinatorial properties of the graph.
The eigenvalues of $M$ along with their multiplicities are called the $M$-\emph{spectrum} of $G$. 
We say that two graphs are $M$-\emph{cospectral} if they have the same $M$-spectrum.
A graph $G$ is \emph{determined by its $M$-spectrum} (or $M$-DS for short) if the only graphs which are $M$-cospectral with $G$ are isomorphic to $G$.

The graph isomorphism problem has motivated the interest in understanding  which graphs are uniquely determined by its $M$-spectrum.
Haemers conjectured that the fraction of graphs on $n$ vertices with a $M$-cospectral mate (i.e., a non-isomorphic graph which is $M$-cospectral) tends to zero as $n$ tends to infinity.
Godsil and McKay \cite{gm} performed a numerical study for cospectral graphs with $n\leq 9$ vertices with respect adjacency matrix. 
Haemers and Spence \cite{HS2004} continued the study for $n = 10, 11$.
Later, Brouwer and Spence \cite{bs} complemented the study for $n=12$. 
Recently, Aouchiche and Hansen \cite{ah} presented computational results in which they studied cospectrality for the distance, distance Laplacian and distance signless Laplacian matrices of all the connected graphs on up to 10 vertices. 

An important question is whether it is possible to define a matrix $M$ of $G$ such that every graph becomes $M$-DS. 
In \cite[Section 2.5]{vDH}, it was shown that the answer to this question is positive. However, in this case it is more difficult to check cospectrality of these matrices
than testing isomorphism. 
If there would be an easily (polynomial-time) computable matrix $M$ for which every graph becomes $M$-DS, then the graph isomorphism problem would be solved. 
One can say that $M$  is not such a matrix for which all graphs are $M$-DS
when $M$ is one of the commonly used matrices associated with graphs (adjacency, Laplacian, distance matrices, signless Laplacian, normalized Laplacian), since there exist many examples of non-isomorphic graphs that share the same $M$-spectrum.
This leaves open the possibility of amplifying or replacing spectra with the use of more refined representations for obtaining more faithful graph information.

With this in mind, we would like to do a similar analysis for the Smith normal which we now recall.
Two $n\times n$ matrices $M$ and $N$ are \emph{equivalent}, denoted by $M\sim N$, if there exist $P,Q\in GL_n(\mathbb{Z})$ such that $N=PMQ$.
That is, $M$ can be transformed to $N$ by applying the following elementary row and column operations which are invertible over the ring of integers:
\begin{enumerate}
  \item Swapping any two rows or any two columns.
  \item Adding integer multiples of one row/column to another row/column.
  \item Multiplying any row/column by $\pm 1$.
\end{enumerate}

Given a square integer matrix $M$, the Smith normal form (SNF) of $M$ is the unique equivalent diagonal matrix $\diag(f_1,f_2,\dots,f_r,0,\dots,0)$ whose non-zero entries are non-negative and satisfy $f_i$ divides $f_{i+1}$, and $r$ is the rank of $M$. 
The diagonal elements of the SNF are known as \emph{invariant factors} or \emph{elementary divisors} of $M$.

The SNF of integer matrices of graphs have been of great interest since it describes the Abelian group obtained from the cokernel.
If we consider an $m\times n$ matrix $M$ with integer entries as a linear map $M:\mathbb{Z}^n\rightarrow \mathbb{Z}^m$, the \emph{cokernel} of $M$ is the quotient module $\mathbb{Z}^{m}/\im(M)$.
The structure theorem for finitely generated Abelian groups implies that the cokernel of $M$ can be described as:
$\coker(M)\cong \mathbb{Z}_{f_1(M)} \oplus \mathbb{Z}_{f_2(M)} \oplus \cdots \oplus\mathbb{Z}_{f_{r}(M)} \oplus \mathbb{Z}^{m-r}$,
where $r$ is the rank of $M$ and $f_1(M), f_2(M), \ldots, f_{r}(M)$ are the invariant factors  of $M$.
This finitely generated Abelian group becomes a graph invariant when we take the matrix $M$ to be a matrix associated with the graph.
The cokernel of the adjacency matrix $A(G)$ is known as the \emph{Smith group} of $G$, and the torsion part of the cokernel of the Laplacian matrix $L(G)$ is known as the \emph{critical group} of $G$ (also known as \emph{sandpile group} of $G$).
For the reader interested in the sandpile group, we recommend the book \cite{Klivans}.
Several sandpile groups have been computed, see \cite{alfaval0} and its references.
Also, the SNF of distance, distance Laplacian, and walk matrices have been studied.
In \cite{HW}, the SNF of the distance matrices were determined for trees, wheels, cycles, and complements of cycles and were partially determined for complete multipartite graphs. 
In \cite{BK}, the SNF of the distance matrices of unicyclic graphs and of the wheel graph with trees attached to each vertex were obtained.
An account on the Smith normal form in combinatorics can be found in \cite{stanley}.

For each graph $G$, let $M(G)$ be a symmetric matrix with integer entries. 
We say that the graphs $G$ and $H$ are $M$-\emph{coinvariant} if the SNFs of $M(G)$ and $M(H)$, computed over $\mathbb{Z}$, are the same.
Coinvariant graphs were introduced in \cite{vince}.
Later, Abiad and Alfaro \cite{aa} started to study whether the portion of graphs that have a $M$-coinvariant mate is smaller than the portion of graphs having a $M$-cospectral mate for a particular matrix $M$. 
Cospectrality and coinvariancy both play an important role in the famous graph isomorphism problem. While it is unknown whether testing graph isomorphism is a hard problem or not, determining whether two graphs are cospectral or coinvariant can be done in polynomial time \cite{kannan,SNFP}. It is also known that testing coinvariancy is experimentally faster than testing cospectrality \cite{akm}. Thus one can focus on testing isomorphism among coinvariant graphs.

In this paper we investigate how good are the spectrum and the SNF of the matrices $A^{\tr}$, $A^{\tr}_+$, $D^{\deg}_+$ and $D^{\deg}$ for determining graphs.
We found that $A^{\tr}$ performs better when compared with previous explored matrices by Abiad and Alfaro \cite{aa}.
In Section~\ref{sec:CoinvariantGraphs}, we extend previous enumeration results to the spectrum and the SNF of matrices $A^{\tr}$, $A^{\tr}_+$, $D^{\deg}_+$ and $D^{\deg}$ of trees with up to 20 vertices.
We found no pair of cospectral trees with respect $D^{\deg}_+$, $D^{\deg}$ and $A^{\tr}$ with up to 20 vertices.
Neither a pair of coinvariant trees with respect $A^{\tr}_+$ and $A^{\tr}$ with up to 20 vertices.
Therefore, it is conjectured that trees are determined by spectrum of $D^{\deg}_+$, $D^{\deg}$ and $A^{\tr}$, and determined by the SNF of $A^{\tr}_+$ and $A^{\tr}$.
Then, in Section~\ref{sec:SNF} we outline how the SNF of $A^{\tr}$ could be used to compute the structure of the sandpile group of certain graphs.
We compute the SNF of $D^{\deg}_+$, $D^{\deg}$, $A^{\tr}_+$ and $A^{\tr}$ for several graph families.
Also, we prove that complete graphs are determined by the SNF of $D^{\deg}_+$, $D^{\deg}$, $A^{\tr}_+$ and $A^{\tr}$.
Finally, in Section~\ref{sec:spectrum} 
we study the spectrum of $A^{\tr}$ and $D^{\deg}$. 
We show that for distance-regular graphs, the spectrum of $A^{\tr}$ is determined by the spectrum of the adjacency matrix and the spectrum of $D^{\deg}$ is determined by the spectrum of the distance matrix. 
We also provide combinatorial bounds for the two smallest eigenvalues of $A^{\tr}$
and give results about their multiplicity.

\section{Cospectral and Coinvariant Graphs}\label{sec:CoinvariantGraphs}

Since we will use several graph distance matrices, we focus on connected graphs so that our enumeration results are comparable. 
Given a connected graph $G$, aside from $A(G)$, $D(G)$, $A^{\tr}(G)$, $A^{\tr}_+(G)$, $D^{\deg}_+(G)$ and $D^{\deg}(G)$, we will use the following matrices: the Laplacian matrix $L(G)$, the signless Laplacian matrix $Q(G)$, the distance Laplacian matrix $D^L(G)$ and the distance signless Laplacian matrix $D^Q(G)$. 
In \cite{aa}, enumeration results were obtained for the spectrum and the invariant factors.
We recall some of these results for the matrices $A$, $L$, $Q$, $D$, $D^L$ and $D^Q$. 
Then, we compare them with the ones obtained for the matrices $A^{\tr}$, $A^{\tr}_+$, $D^{\deg}$ and $D^{\deg}_+$. 
We performed our computations over Sagemath \cite{sage}.

\begin{table}[ht]
    \centering
	{\small
    \begin{tabular}{cccccccc}
		\hline
        $n$ & 4 & 5 & 6 & 7 & 8 & 9 & 10\\
        \hline
        $|\mathcal{G}_n|$ & 6 & 21 & 112 & 853 & 11,117 & 261,080 & 11,716,571\\
        \hline
        $|\mathcal{G}^{sp}_n(A)|$ & 0 & 0 & 2 & 63 & 1,353 & 46,930 & 2,462,141 \\
        $|\mathcal{G}^{sp}_n(L)|$ & 0 & 0 & 4 & 115 & 1,611 & 40,560 & 1,367,215 \\
        $|\mathcal{G}^{sp}_n(Q)|$ & 0 & 2 & 10 & 80 & 1,047 & 17,627 & 615,919  \\
        $|\mathcal{G}^{sp}_n(D)|$ & 0 & 0 & 0 & 22 & 658 & 25,058 & 1,389,986   \\
        $|\mathcal{G}^{sp}_n(D^L)|$ & 0 & 0 & 0 & 43 & 745 & 20,455 & 787,851\\
        $|\mathcal{G}^{sp}_n(D^Q)|$ & 0 & 2 & 6 & 38 & 453 & 8,168 & 319,324\\
        \hline
        $|\mathcal{G}^{in}_n(A)|$ & 4 & 20 & 112 & 853 & 11,117 & 261,080 & 11,716,571 \\
        $|\mathcal{G}^{in}_n(L)|$ & 2 & 8 & 57 & 526 & 8,027 & 221,834 & 11,036,261 \\
        $|\mathcal{G}^{in}_n(Q)|$ & 2 & 11 & 78 & 620 & 7,962 & 201,282 & 10,086,812 \\
        $|\mathcal{G}^{in}_n(D)|$ & 2 & 15 & 102 & 835 & 11,080 & 260,991 & 11,716,249\\
        $|\mathcal{G}^{in}_n(D^L)|$  & 0 &   0 &   0 &  18 &   455 & 16,505 & 642,002\\
        $|\mathcal{G}^{in}_n(D^Q)|$ & 0 &   2 & 4 & 20 & 259 & 7,444 &  264,955\\
        \hline
	\end{tabular}
	}
\caption{Number of connected graphs with at least one cospectral mate and one coinvariant mate for $A$, $L$, $Q$, $D$, $D^L$ and $D^Q$.}
	\label{Tab:cospectralcoinvariantALQDDLDQ}
\end{table}

Denote by $\mathcal{G}_n$ the set of connected graphs with $n$ vertices.
Let $\mathcal{G}^{sp}_n(M)$ be the set of graphs in $\mathcal{G}_n$ which have at least one cospectral mate in $\mathcal{G}_n$ with respect to the matrix $M$.
Similarly, let $\mathcal{G}^{in}_n(M)$ be the set of graphs in $\mathcal{G}_n$ which have at least one coinvariant mate in $\mathcal{G}_n$ with respect to the matrix $M$.
Table~\ref{Tab:cospectralcoinvariantALQDDLDQ} shows the enumeration of $\mathcal{G}^{sp}_n(M)$ and $\mathcal{G}^{in}_n(M)$ for several  matrices.
From this table, it follows that the spectrum and the Smith normal form of $D^Q$ could be the best for distinguishing connected graphs.

A lot of research has been devoted to understand cospectral graphs, but much less has been dedicated to understand the coinvariance relationship and its potential to characterize graphs.
One reason for this could be that there is a large proportion of connected graphs having a $M$-coinvariant mate with respect to the matrices $A$, $L$, $Q$ and $D$, as Table \ref{Tab:cospectralcoinvariantALQDDLDQ} shows. 
However, in \cite{aa}, it was discovered the potential of using the SNF of $D^L$ and $D^Q$ to characterize graphs.

\begin{table}[ht]
    \centering
	{\small
    \begin{tabular}{cccccccc}
		\hline
        $n$ & 4 & 5 & 6 & 7 & 8 & 9 & 10\\
        \hline
        $|\mathcal{G}_n|$ & 6 & 21 & 112 & 853 & 11,117 & 261,080 & 11,716,571\\
        \hline
        $|\mathcal{G}^{sp}_n(D^{\deg})|$ & 0 & 2 & 6 & 40 & 485 & 9,784 & 355,771 \\

        $|\mathcal{G}^{sp}_n(D^{\deg}_+)|$ & 0 & 0 & 0 & 61 & 901 & 24,095 & 852,504 \\

        $|\mathcal{G}^{sp}_n(A^{\tr})|$ & 0 & 2 & 6 & 38 & 413 & 7,877 & 299,931 \\

        $|\mathcal{G}^{sp}_n(A^{\tr}_+)|$ & 0 & 0 & 0 & 43 & 728 & 19,757 & 765,421 \\
        \hline
        $|\mathcal{G}^{in}_n(D^{\deg})|$  &  2  &  2  &  6  &  34 & 538  & 17,497  & 902,773 \\
        $|\mathcal{G}^{in}_n(D^{\deg}_+)|$   &  2  & 11  & 46  & 495 & 7,169 & 209,822 & 10,815,879 \\
        $|\mathcal{G}^{in}_n(A^{\tr})|$   &  0  &  2  &  4  & 22  & 240  & 6,642   & 237,118 \\
        $|\mathcal{G}^{in}_n(A^{\tr}_+)|$     &  0  &  0  &  0  & 16  & 456  & 15,952  & 605,625  \\
        \hline
	\end{tabular}
	}
\caption{Number of connected graphs with at least one cospectral mate and one coinvariant mate for $A^{\tr}(G)$, $A^{\tr}_+(G)$, $D^{\deg}_+(G)$ and $D^{\deg}(G)$.}
	\label{Tab:coinvariantAtrDdeg}
\end{table}

In Table~\ref{Tab:coinvariantAtrDdeg}, it is shown the number of connected graphs with at least one cospectral mate and one coinvariant mate for the matrices $A^{\tr}$, $A^{\tr}_+$, $D^{\deg}_+$ and $D^{\deg}$.
In order to do a simpler comparison between previous and new results, we follow \cite{aa} in defining the \emph{spectral uncertainty} $\SP_n(M)$ with respect to $M$ as the ratio $|\mathcal{G}^{sp}_n(M)|/|\mathcal{G}_n|$, and the \emph{invariant uncertainty} $\IN_n(M)$ with respect to $M$ as the ratio $|\mathcal{G}^{in}_n(M)|/|\mathcal{G}_n|$.
Figure~\ref{fig:spectruminvariantQDLDQ} displays the spectral and the invariant uncertainty for the matrices that better determine graphs.
We only show the five best parameters to distinguish graphs.
According to our results, the SNF of $A^{\tr}$ performs better than the rest of the parameters for distinguishing graphs for all considered matrices.
This is followed by the SNF of $D^Q$ and the spectrum of $A^{\tr}$, $D^Q$ and $D^{\deg}$, respectively.

\begin{figure}[ht]
    \centering
    \begin{tikzpicture}[trim axis left]
    \begin{axis}[
        scale only axis,
        title={},
        xlabel={$n$},
        width=8cm, 
        height=5cm,
        xmin=4, xmax=10,
        ymin=0.0, ymax=0.1,
        xtick={4,5,6,7,8,9,10},
        legend style ={ 
            row sep=-0.2cm,
            at={(1.03,1)}, 
            anchor=north west,
            draw=black, 
            fill=white,
            align=left
        },
        ymajorgrids=true,
        grid style=dashed,
        legend columns=1
    ]
    \addplot [
        dashed,
        every mark/.append style={solid, fill=gray},
        mark=square*,
        legend entry={\tiny $\SP_n(D^{\deg})$}
        ]
        coordinates {
        (4,0.0) (5,0.09523809523809523) (6,0.05357142857142857) (7,0.04689331770222743) (8,0.04362687775478996) (9,0.037475103416577296) (10,0.030364771399413702)
        };

    \addplot [
        dashed,
        every mark/.append style={solid, fill=gray},
        mark=diamond*,
        legend entry={\tiny $\SP_n(A^{\tr})$}
        ]
        coordinates {
        (4,0.0) (5,0.09523809523809523) (6,0.05357142857142857) (7,0.044548651817116064) (8,0.03715031033552217) (9,0.030170828864715796) (10,0.025598871888370754)
        };
        
    \addplot[
        dashed,
        every mark/.append style={solid, fill=gray},
        mark=*,
        legend entry={\tiny $\SP_n(D^Q)$}
        ]
        coordinates {
        (4,0.0) (5,0.09523809523809523) (6,0.05357142857142857) (7,0.044548651817116064) (8,0.0407484033462265) (9,0.03128542975333231) (10,0.02725404898754081)
        };
     
     \addplot [
        dashed,
        every mark/.append style={solid, fill=white},
        mark=*,
        legend entry={\tiny $\IN_n(D^Q)$}
        ]
        coordinates {
        (4,0.0) (5,0.09523809523809523) (6,0.03571428571428571) (7,0.023446658851113716) (8,0.023297652244310515) (9,0.02851233338440325) (10,0.022613698154519784)
        };

    \addplot [
        dashed,
        every mark/.append style={solid, fill=white},
        mark=diamond*,
        legend entry={\tiny $\IN_n(A^{\tr})$}
        ]
        coordinates {
        (4,0.0) (5,0.09523809523809523) (6,0.03571428571428571) (7,0.02579132473622509) (8,0.02158855806422596) (9,0.025440478014401715) (10,0.020237832382870382)
        };
    
    \end{axis}
    \end{tikzpicture}
    \caption{The fraction of graphs on $n$ vertices that have at least one cospectral mate with respect to a certain matrix is denoted as $sp$. The fraction of graphs on $n$ vertices with respect to a certain matrix that have at least one coinvariant mate is denoted as $in$.}
    \label{fig:spectruminvariantQDLDQ}
\end{figure}
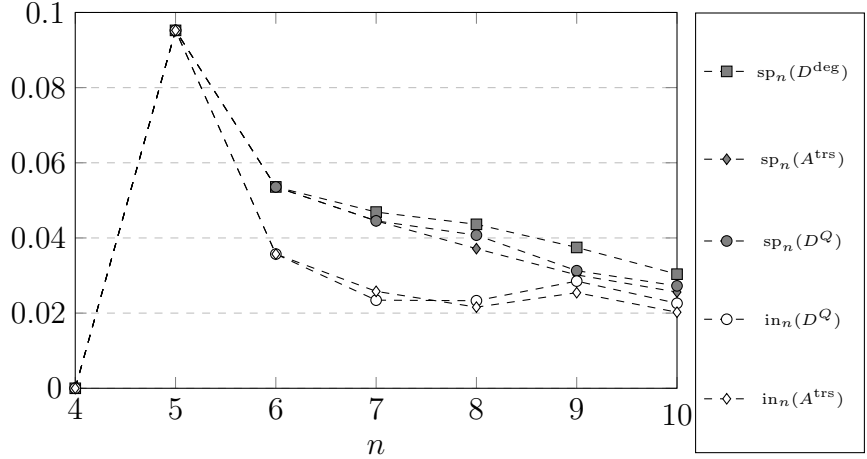

\subsection{Cospectral and Coinvariant Trees}
We end up this section with an observation on cospectral and coinvariant trees. 
Aouchiche and Hansen \cite{ah} reported enumeration results on cospectral trees with at most 20 vertices with respect to $D$, $D^L$ and $D^Q$.
For $D$, they found that among the 123,867 trees on 18 vertices, there are two pairs of $D$-cospectral mates.
Among the 317,955 trees on 19 vertices, there are six pairs of $D$-cospectral mates.
There are 14 pairs of $D$-cospectral mates over all the 823,065 trees on 20 vertices.
And surprisingly, after the enumeration of all 1,346,023 trees on at most 20 vertices, they found no $D^L$-cospectral mates and no $D^Q$-cospectral mates.
This fact lead Aouchiche and Hansen to conjecture that every tree is determined by its distance Laplacian spectrum, and by its distance signless Laplacian spectrum.

Analogously, Abiad and Alfaro \cite{aa} reported enumeration results for the SNF of $D$, $D^L$ and $D^Q$ of trees. 
Hou and Woo \cite{HW} extended Graham and Pollak celebrated formula, that states $\det(D(T_{n+1}))=(-1)^nn2^{n-1}$ for any tree $T_{n+1}$ with $n+1$ vertices, to obtain $\SNF(D(T_{n+1}))=\sf{I}_2\oplus 2\sf{I}_{n-2}\oplus (2n)$.
This implies that all trees on $n$ vertices have the same SNF of its distance matrix $D$.
After enumerating coinvariant trees with at most 20 vertices with respect to $D^L$ and $D^Q$, they found no $D^L$-coinvariant mates and no $D^Q$-coinvariant mates among all trees with up to 20 vertices.
Thus, it is also conjectured that almost all trees are determined by the SNF of its $D^L$, and analogously, by the SNF of its $D^Q$.

We extended previous enumeration results to the spectrum and the SNF of matrices $A^{\tr}$, $A^{\tr}_+$, $D^{\deg}_+$ and $D^{\deg}$ of trees with up to 20 vertices.
Our findings are the following.
There are no $D^{\deg}_+$-cospectral trees nor $D^{\deg}$-cospectral trees nor $A^{\tr}$-cospectral trees with up to 20 vertices.
However, there are four $A^{\tr}_+$-cospectral trees with 20 vertices and no $A^{\tr}_+$-cospectral trees with 19 vertices.

On the other hand, there are no $A^{\tr}_+$-coinvariant trees nor $A^{\tr}$-coinvariant trees with up to 20 vertices.
However, we found two $D^{\deg}$-coinvariant trees with 14 vertices,
6 $D^{\deg}$-coinvariant trees with 16 vertices,
14 $D^{\deg}$-coinvariant trees with 17 vertices,
22 $D^{\deg}$-coinvariant trees with 18 vertices,
40 $D^{\deg}$-coinvariant trees with 19 vertices, and 
76 $D^{\deg}$-coinvariant trees with 20 vertices.

With respect to the matrix $D^{\deg}_+$, we found two $D^{\deg}_+$-coinvariant trees with 9 vertices, 
6 $D^{\deg}_+$-coinvariant trees with 10 vertices, 
20 $D^{\deg}_+$-coinvariant trees with 11 vertices, 
46 $D^{\deg}_+$-coinvariant trees with 12 vertices, 
148 $D^{\deg}_+$-coinvariant trees with 13 vertices, 
373 $D^{\deg}_+$-coinvariant trees with 14 vertices, 
1093 $D^{\deg}_+$-coinvariant trees with 15 vertices, 
2912 $D^{\deg}_+$-coinvariant trees with 16 vertices, 
8189 $D^{\deg}_+$-coinvariant trees with 17 vertices, 
22551 $D^{\deg}_+$-coinvariant trees with 18 vertices, 
64738 $D^{\deg}_+$-coinvariant trees with 19 vertices, and 
183211 $D^{\deg}_+$-coinvariant trees with 20 vertices.

Considering this evidence, it is possible that trees are determined by spectrum of $D^{\deg}_+$, $D^{\deg}$ and $A^{\tr}$, and also determined by the SNF of $A^{\tr}_+$ and $A^{\tr}$.

\section{The SNF of $A^{\tr}$, $A^{\tr}_+$, $D^{\deg}_+$ and $D^{\deg}$}\label{sec:SNF}

The SNF of the Laplacian matrix is commonly used to compute the algebraic structure of the sandpile group, since it coincides with the torsion part of the cokernel of the Laplacian matrix.
Let $L^q(G)$ be the Laplacian matrix of a connected graph $G$, where the column and row associated with a vertex $q\in V(G)$ are removed.
This matrix is known as a \emph{reduced Laplacian}, and its SNF also keeps the structure of the sandpile but without the torsion part, for this reason the reduced Laplacian is also used for computing the sandpile group of connected graphs.
It is interesting to note that $A^{\tr}$ can be regarded as the reduced Laplacian of a graph as follows.
We have that $\deg(v)\leq \tr(v)$ for any $v\in V(G)$, and  equality holds when $v$ is adjacent with each vertex $u\in V(G)\setminus\{v\}$.
Let $G$ be a connected graph different from the complete graph.
Suppose $H$ is the graph $G$ with an additional vertex $q$, such that $q$ is adjacent with each $v\in V(G)$ by mean of $\tr(v)-\deg(v)$ edges.
Since $G$ is not the complete graph there is at least one vertex with $\tr(v)-\deg(v)>0$.
Therefore, $A^{\tr}(G)=L^q(H)$ and the transmission-adjacency matrix is an additional matrix that can be used to compute the sandpile group and the number of spanning trees of $H$.

\begin{theorem}
    Let $G$ be a connected graph with $n$ vertices, different from the complete graph, and let $H$ be the graph $G$ with an additional vertex $q$, such that $q$ is adjacent with each $v\in V(G)$ by mean of $\tr(v)-\deg(v)$ edges.
    Then, if the Smith normal form of $A^{\tr}$ is $\diag(f_1,\dots,f_n)$, then the sandpile group $K(H)$ of $H$ is isomorphic to $\mathbb{Z}_{f_1}\oplus\cdots\oplus\mathbb{Z}_{f_n}$ and the number of spanning trees $\tau(H)$ of $H$ equals $\prod_{i=1}^nf_i$.
\end{theorem}

This kind of connection is useful for computing the algebraic structure of the sandpile group.
For instance, in \cite{alfavilla} it was used the cycle-intersection matrix for computing the structure of the sandpile groups of outerplanar graphs.

\begin{example}
    Let $\ltimes$ be the cricket graph.
    Let $H$ be the graph $\ltimes$ with an additional vertex $q$, such that $q$ is adjacent with each $v\in V(\ltimes)$ by mean of $\tr(v)-\deg(v)$ edges, such as in Figure~\ref{fig:cricketgraph}.
    Then
    \[
    A^{\tr}(\ltimes)=
     \begin{bmatrix}
         6 &  0 &  0 & -1 & -1\\
         0 &  7 &  0 &  0 & -1\\
         0 &  0 &  7 &  0 & -1\\
        -1 &  0 &  0 &  6 & -1\\
        -1 & -1 & -1 & -1 &  4
     \end{bmatrix}
    \]
    and
    \[
    L(H)=
     \begin{bmatrix}
         6 &  0 &  0 & -1 & -1 & -4\\
         0 &  7 &  0 &  0 & -1 & -6\\
         0 &  0 &  7 &  0 & -1 & -6\\
        -1 &  0 &  0 &  6 & -1 & -4\\
        -1 & -1 & -1 & -1 &  4 & 0\\
        -4 & -6 & -6 & -6 & -4 & 20
     \end{bmatrix}.
    \]
    From this it follows that the SNF of $A^{\tr}(\ltimes)$ is $\diag(1, 1, 1, 7, 812)$ and the SNF of $L(H)$ is $\diag(1, 1, 1, 7, 812,0)$.
    Thus the sandpile group of $H$ is isomorphic to $\mathbb{Z}_7\oplus\mathbb{Z}_{812}$.
\end{example}

\begin{figure}[ht]
    \centering
    \begin{tabular}{c@{\extracolsep{1.5cm}}c}
        \begin{tikzpicture}[rotate=90,scale=.7,thick]
    	\tikzstyle{every node}=[minimum width=0pt, inner sep=2pt, circle]
            \clip (1.6,1.5) rectangle (-1.6,-1.5);
    	\draw (-.5,-.9) node (v1) [draw] {\tiny 4};
    	\draw (.5,-.9) node (v2) [draw] {\tiny 1};
    	\draw (0,0) node (v3) [draw] {\tiny 5};
    	\draw (-.5,.9) node (v4) [draw] {\tiny 3};
    	\draw (.5,.9) node (v5) [draw] {\tiny 2};
    	\draw (v1) -- (v2);
    	\draw (v1) -- (v3);
    	\draw (v2) -- (v3);
    	\draw (v3) -- (v4);
    	\draw (v3) -- (v5);
        \end{tikzpicture}
         &  
         \begin{tikzpicture}[rotate=90,scale=.7,thick]
            \clip (1.6,3.5) rectangle (-1.6,-1.5);
    	\tikzstyle{every node}=[minimum width=0pt, inner sep=2pt, circle]
    	\draw (-.5,-.9) node (v1) [draw] {};
    	\draw (.5,-.9) node (v2) [draw] {};
    	\draw (0,0) node (v3) [draw] {};
    	\draw (-.5,.9) node (v4) [draw] {};
    	\draw (.5,.9) node (v5) [draw] {};
            \draw (0,3) node (q) [draw] {};
    	\draw (v1) -- (v2);
    	\draw (v1) -- (v3);
    	\draw (v2) -- (v3);
    	\draw (v3) -- (v4);
    	\draw (v3) -- (v5);
            \draw (q) edge[bend right=45] (v1);
            \draw (q) edge[bend right=55] (v1);
            \draw (q) edge[bend right=65] (v1);
            \draw (q) edge[bend right=75] (v1);
            \draw (q) edge[bend left=45] (v2);
            \draw (q) edge[bend left=55] (v2);
            \draw (q) edge[bend left=65] (v2);
            \draw (q) edge[bend left=75] (v2);
            \draw (q) edge[bend right=10] (v4);
            \draw (q) edge[bend right=20] (v4);
            \draw (q) edge[bend right=30] (v4);
            \draw (q) edge[bend left=10] (v4);
            \draw (q) edge[bend left=20] (v4);
            \draw (q) edge (v4);
            \draw (q) edge[bend right=10] (v5);
            \draw (q) edge[bend right=20] (v5);
            \draw (q) edge[bend left=30] (v5);
            \draw (q) edge[bend left=10] (v5);
            \draw (q) edge[bend left=20] (v5);
            \draw (q) edge (v5);
        \end{tikzpicture}
        \\
        $\ltimes$ & $H$
    \end{tabular}
    
    \caption{A super graph constructed from the cricket graph.}
    \label{fig:cricketgraph}
\end{figure}
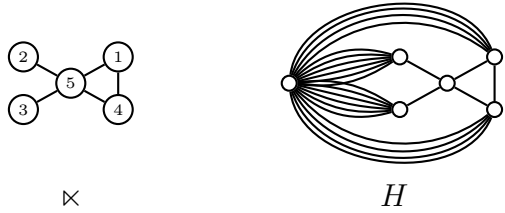

Let us now introduce the notion of determinantal ideals of graphs which will be used to compute the SNF and prove that complete graphs are determined by the SNF of $A^{\tr}$, $A^{\tr}_+$, $D^{\deg}$ and $A^{\deg}_+$.
Consider an $n\times m$ matrix $M$ whose entries are in the polynomial ring $\mathbb{Z}[X]$ with $X$ a set of $l$ indeterminates $x_1, \dots, x_l$.
We will assume that $n\leq m$ to simplify notation, because otherwise we can transpose the matrix without changing the determinants of its sub-matrices.
For $k\in [n]:=\{1, \dots, n\}$, let $\mathcal{I}=\{r_j\}_{j=1}^k$ and $\mathcal{J}=\{c_j\}_{j=1}^k$ be two sequences such that
\[
1\leq r_1 < r_2 < \cdots < r_k \leq n \text{ and } 1\leq c_1 < c_2 < \cdots < c_k \leq m.
\]
Let $M[\mathcal{I;J}]$ denote the submatrix of a matrix $M$ induced by the rows with indices in $\mathcal{I}$ and columns with indices in $\mathcal{J}$.
The determinant of $M[\mathcal{I;J}]$ is called a $k$-\emph{minor} of $M$.
We denote by $\minors_k(M)$ the set of all $k$-minors of $M$.

\begin{definition}
For $k\in[n]$, the $k$-\emph{th} \emph{determinantal ideal of a matrix}  $M$ with entries in $\mathbb{Z}[X]$, denoted $I_k(M)$, is the ideal generated by $\minors_k(M)$.
\end{definition}

Let $I\subseteq \mathbb{Z}[X]$ be an ideal.
The \emph{variety} $V(I)$ of $I$ is defined as the set of common roots between polynomials in $I$.
In several contexts it will be more convenient to consider an extension $\mathcal{P}$ of $\mathbb{Z}$ to define the variety $$V^\mathcal{P}(I):=\left\{ {\bf a}\in \mathcal{P}^l : f({\bf a}) = 0 \text{ for all } f\in I \right\}.$$

The following result shows the contention between ideals and varieties.

\begin{proposition}\cite{Nbook}\label{prop:chaininclusionideals}
Let $M$ be an $n\times m$ matrix with entries in $\mathbb{Z}[X]$.
Then, it holds that
\[
\langle 1\rangle \supseteq I_1(M) \supseteq \cdots \supseteq I_n(M) \supseteq \langle 0\rangle
\]
and
\[
\emptyset\subseteq V^\mathcal{P}(I_1(M)) \subseteq \cdots \subseteq V^\mathcal{P}(I_n(M)) \subseteq \mathcal{P}^l.
\]
\end{proposition}

An ideal is said to be \emph{trivial} or \emph{unit} if it is equal to $\langle1\rangle$, which is in fact equal to $\mathbb{Z}[X]$. The \emph{algebraic co-rank} $\gamma(M)$ of $M$ is the maximum integer $k$ for which $I_k(M)$ is trivial.

\begin{definition}\label{def:polynomialmat}
Given a graph $G$ with $n$ vertices and a set of indeterminates $X=\{x_u \, : \, u\in V(G)\}$, we define the following polynomial matrices:
\begin{itemize}
    \item $A_X(G):=\diag(x_1, \dots,x_n)-A(G)$,
    \item $D_X(G):=\diag(x_1, \dots,x_n)+D(G)$.
\end{itemize}
\end{definition}

The determinantal ideals associated with the matrices $A_X$ and $D_X$ have been studied in \cite{akm,alfaro2,alflin,at,corrval} under the name of \emph{critical ideals} and \emph{distance ideals}, respectively.

A useful way to compute the invariant factors is given by the following result, which we shall use to connect determinatal ideals with SNF.

\begin{theorem}[Elementary divisors theorem]\label{theo:edt}\cite{jacobson}
Let $M$ be an $n\times m$ matrix with entries in a \emph{principal ideal domain} (PID).
Then the $k$-\emph{th} invariant factor $f_k(M)$ of $M$ is equal to $\Delta_k(M)/ \Delta_{k-1}(M)$, where $\Delta_k(M)$ is the \emph{greatest common divisor} of the $k$-minors of $M$ and $\Delta_0(M)=1$.
\end{theorem}

Let $M$ be a $n\times n$ integer matrix and let $M_X$ be the polynomial matrix $\diag(x_1,\dots,x_n)-M$.
We can recover the invariant factors of $M$ from the determinantal ideals of $M_X$ as described in the following result.
\begin{proposition}\label{prop:evalmultiplevariables}\cite{akm}
Let $M$ be an $n\times n$ matrix with entries in $\mathbb{Z}$, ${\bf c}$ a row vector in $\mathbb{Z}^n$ and $X=\{x_1,\dots,x_n\}$ a set of indeterminates.
Let $M_X=\diag(x_1,\dots,x_n)-M$.
Then the ideal $I^\mathcal{R}_k(M_X)$ evaluated at $X={\bf c}$ is generated by $\Delta_k({\bf c}I_n-M)$, the \emph{gcd} of the $k$-minors of ${\bf c}I_n-M$ over $\mathbb{Z}$.
\end{proposition}

We can apply the previous result, for example, by evaluating the determinantal ideals of $A_X(G)$ at $X=\deg(G)$ to recover the SNF of the Laplacian matrix.

\begin{example}\label{example:grobner1}
Consider the cycle with 5 vertices and
\[
A_X(C_5)=
\begin{bmatrix}
    x_0 & -1 &  0 &  0 &  -1\\
    -1 & x_1 & -1 &  0 &   0\\
     0 & -1 & x_2 & -1 &   0\\
     0 &  0 & -1 & x_3 &  -1\\
    -1 &  0 &  0 &  -1 & x_4\\
\end{bmatrix}.
\]
In what follows, we give a Gr\"obner bases of the critical ideals over $\mathbb{Z}[X]$.
Note $I_1(A_X(C_5))$, $I_2(A_X(C_5))$ and $I_3(A_X(C_5))$ are trivial.
The fourth critical ideal of $C_5$ is equal to
\[
\langle x_0x_1 + x_3 - 1, x_1x_2 + x_4 - 1, x_2x_3 + x_0 - 1, x_0x_4 + x_2 - 1, x_3x_4 + x_1 - 1\rangle.
\]
Meanwhile, the fifth critical ideal of $C_5$ is generated by the polynomial
\[
x_0x_1x_2x_3x_4 - x_0x_1x_2 - x_1x_2x_3 - x_0x_1x_4 - x_0x_3x_4 - x_2x_3x_4 + x_0 + x_1 + x_2 + x_3 + x_4 - 2.
\]
By evaluating the critical ideals of $C_5$ at $(0,0,0,0,0)$, we obtain that the SNF of $A(C_5)$ is $\diag(1,1,1,1,2)$.
On the other hand, if we evaluate the critical ideals at \[(\deg(v_1),\deg(v_2),\deg(v_3),\deg(v_4))=(2,2,2,2,2),\] we obtain that the SNF of $L(C_4)$ is $\diag(1,1,1,5,0)$.
Similarly, evaluating the critical ideals of $C_4$ at $$(\tr(v_1),\tr(v_2),\tr(v_3),\tr(v_4))=(6,6,6,6,6)$$ and $-(\tr(v_1),\tr(v_2),\tr(v_3),\tr(v_4))$, we obtain that the SNF of $A^{\tr}(C_4)$ and $A^{\tr}_+(C_4)$ are $\diag(1, 1, 1, 41, 164)$ and $(1, 1, 1, 29, 232)$, respectively.
\end{example}

Similarly, distance ideals can be used to compute the SNF of the matrices $D$, $D^L$, $D^Q$, $D^{\deg}$ and $D^{\deg}_+$.
Let us recall the critical ideals of complete graphs to compute the SNF of matrices of complete graphs.

\begin{theorem}\cite{corrval}\label{teo:criticalidealscompletegraph}
    Let $K_n$ be the complete graph with $n\geq2$ vertices and $1\leq m\leq n-1$, then
    \[
        B_m= \left\{ \prod_{i \in \mathcal{I}}(x_i+1) \, \big| \, \mathcal{I}\subset [n], \, |\mathcal{I}|=m-1 \right\}
    \]
    is a reduced Gr\"obner basis of $I_m\left(A_X(K_n)\right)$ with respect to the graded lexicographic order.
    And
    \[
        \det(A_X(K_n))=\prod_{i=1}^n(x_i+1)-\sum_{j=1}^n\prod_{i\neq j}(x_i+1)
    \]
\end{theorem}

One particular determinantal ideal is the case when all variables are equal to one variable, say $t$, in this case, the matrix $A_t(G)=tI_n-A(G)$ is used instead of $A_X$.
The determinantal ideals of $A_t(G)$ are known as \emph{characteristic ideals}, see \cite{akm,absv,yibo}.
The following corollary follows directly from Theorem~\ref{teo:criticalidealscompletegraph}.

\begin{corollary}
    Let $K_n$ be the complete graph with $n\geq2$ vertices and $1\leq m\leq n-1$, then $I_m(A_t(K_n))\subseteq\mathbb{Z}[t]$ is equal to $\langle (t+1)^{m-1} \rangle$ for $1\leq m \leq n-1$, and $\det(A_t(K_n))=(t+1)^n-n(t+1)^{n-1}$.
\end{corollary}

Note that for the complete graph $K_n$ with $n$ vertices, it can be checked that $A(K_n)=D(K_n)$ and $\deg(K_n)=\tr(K_n)$, and hence $A^{\tr}(K_n)=D^{\deg}(K_n)=L(K_n)$ and $A^{\tr}_+(K_n)=D^{\deg}_+(K_n)=Q(K_n)$.
By product, we can obtain the SNF of $A^{\tr}(K_n)$, $D^{\deg}(K_n)$, $L(K_n)$, $A^{\tr}_+(K_n)$, $D^{\deg}_+(K_n)$ and $Q(K_n)$.

\begin{proposition}\label{prop:SNF_Kn}
    The Smith normal forms of $A^{\tr}(K_n)$, $D^{\deg}(K_n)$ and $L(K_n)$ are equal to $\diag(1,n,\dots,n,0)$.
    The Smith normal forms of $A^{\tr}_+(K_n)$, $D^{\deg}_+(K_n)$ and $Q(K_n)$ are equal to $\diag(1,n-2,\dots,n-2,2(n-1)(n-2))$. 
\end{proposition}
\begin{proof}[\emph{Proof}]
    It follows by evaluating $I_m(A_t(K_n)$ at $t=n-1$ and $t=1-n$.
    Notice that the SNF of a matrix $M$ is equal to the SNF of $-M$.
    Therefore, by Proposition~\ref{prop:evalmultiplevariables}, we obtain the result.
\end{proof} 

Let $K_{m,1}$ be the star graph with $m\geq 1$.
For simplicity, let
\[
    D_X(K_{m,1})=
        \begin{bmatrix}
        \diag(x_1,\dots,x_m)-2{\sf I}_m+2{\sf J}_m & {\sf J}_{m,1}\\
        {\sf J}_{1,m} & y\\
        \end{bmatrix},
\]
where ${\sf I}_m$ and ${\sf J}_m$ denote the identity and the all ones matrices, respectively.
Then, the distance ideal of the star graphs are given by the following result.

\begin{theorem}\cite{at}
    Let
    \[
    C_k=\left\{ \prod\limits_{i\in \mathcal{I}}(x_i-2) : \mathcal{I}\subset [m] \text{ and } |\mathcal{I}| = k-1\right\}
    \]
    and
    \[
    D_k=\left\{ (2y-1)\prod\limits_{i\in \mathcal{I}}(x_i-2) : \mathcal{I}\subset [m] \text{ and } |\mathcal{I}| = k-2  \right\}.
    \]
    For $k \in [m]$, the $k$-th distance ideal of the star graph $K_{m,1}$ is generated by $\langle C_k \cup D_k\rangle\subseteq\mathbb{Z}[x_1,\dots,x_m,y]$.
    And, $\det(D_X(K_{m,1}))$ is equal to
    \begin{equation}
        y\prod\limits_{i=1}^m(x_i-2) + (2y-1)\sum\limits_{i=1}^m\prod\limits^m_{\substack{j=1\\ j\neq i}}(x_j-2)
    \end{equation}
\end{theorem}

By evaluating the distance ideals at the degree vector and at minus the degree vector, we obtain the SNF of $D^{\deg}_+(K_{m,1})$ and $D^{\deg}(K_{m,1})$, respectively.

\begin{proposition}\label{prop:SNFdistancestar}
    The Smith normal form of $D^{\deg}_+(K_{m,1})$ is equal to
    \[
        \diag(1,\dots,1,2m(m-1)),
    \] and the Smith normal form of $D^{\deg}_+(K_{m,1})$ is equal to
    \[
    \begin{cases}
        \diag(1,3,\dots,3,2m(m-1)) & \text{if } 2m + 1 = 3a \text{ for some }a\in\mathbb{Z},\\
        \diag(1,1,3,\dots,3,6m(m-1)) & \text{otherwise.}
    \end{cases}
    \]
\end{proposition}
\begin{proof}[\emph{Proof}]
    For $k\in[m]$, the evaluation of the $k$-\emph{th} distance ideals of $K_{m,1}$ at $(x_1,\dots,x_m,y)=\deg(G)$ yields the ideal $\langle1\rangle$.
    And the evaluation of $\det(D_X(K_{m,1}))$ at $(x_1,\dots,x_m,y)=\deg(G)$ is equal to $m(-1)^m+(2m-1)(-1)^{m-1}m$.
    Then the SNF of $D^{\deg}_+(K_{m,1})$ follows.
    On the other hand, for $k\in[m]$, the evaluation of the $k$-\emph{th} distance ideals of $K_{m,1}$ at $(x_1,\dots,x_m,y)=-\deg(G)$ is equal to $\langle 3^{k-2}\rangle$ or $\langle 3^{k-1}\rangle$ depending whether $2m+1$ is divisible by 3.
    And the evaluation of $\det(D_X(K_{m,1}))$ at $(x_1,\dots,x_m,y)=\deg(G)$ is equal to $m(-3)^m+(-2m-1)(-3)^{m-1}m$.
    From which the SNF of $D^{\deg}(K_{m,1})$ follows.
\end{proof}

Now, we will use some characterizations of critical ideals and distance ideals to prove that complete graphs are determined by the SNF of the matrices $A^{\tr}$, $A^{\tr}_+$, $D^{\deg}$ and $D^{\deg}_+$.

First note that as a consequence of Proposition~\ref{prop:evalmultiplevariables}, for any vector ${\bf d}\in \mathbb{Z}^{V(G)}$, the number of invariant factors equal to 1 of the matrix $M_X$ evaluated at $X={\bf d}$ is at least the number of trivial determinantal ideals of $M_X$.
Therefore, the family of graphs with at most $k$ trivial critical ideals contains the families of graphs whose matrices $A^{\tr}$ and $A_+^{tr}$ have at most $k$ invariant factors equal to 1, and the family of graphs with at most $k$ trivial distance ideals contains the families of graphs whose matrices $D^{\deg}$ and $D_+^{\deg}$ have at most $k$ invariant factors equal to 1.

\begin{theorem}\cite{alfaval}\label{teo:gamma1}
If $G$ is a simple connected graph, then $G$ has only one trivial critical ideal if and only if $G$ is the complete graph.
\end{theorem}

Then the complete graphs are the only connected with $n$ vertices whose matrices $A^{\tr}$ and $A^{\tr}_+$ could contain exactly one invariant factor equal to 1.
And by Proposition~\ref{prop:SNF_Kn}, complete graphs are the only connected graphs with $n$ vertices whose SNF of $A^{\tr}$ has the first invariant factor equal to 1 and the second invariant factor is equal to $n$, and similarly, complete graphs are the only connected graphs with $n$ vertices whose SNF of $A^{\tr}_+$ has the first invariant factor equal to 1 and the second invariant factor is equal to $n-2$.
This implies the following result.

\begin{theorem}
    Complete graphs are determined by the SNF of $A^{\tr}$.
    And complete graphs are determined by the SNF of $A^{\tr}_+$.
\end{theorem}

We now show that complete graphs are determined by the SNF of $D^{\deg}$ and $D^{\deg}_+$.

\begin{theorem}\cite{at}\label{teo:clasificationofgraphswith1trivialdistance}
A connected graph has only one trivial distance ideal over $\mathbb{Z}[X]$ if and only if $G$ is either a complete graph or a complete bipartite graph.
\end{theorem}

We saw in Proposition~\ref{prop:SNF_Kn} that the SNF of $D^{\deg}(K_n)$ is equal to $\diag(1,n,\dots,n,0)$, and the SNF of $D^{\deg}_+(K_n)$ is equal to $\diag(1,n-2,\dots,n-2,2(n-1)(n-2))$.
On the other hand, by Proposition ~\ref{prop:SNFdistancestar}, the SNF of $D^{\deg}_+(K_m,1)= \diag(1,3,\dots,3,2m(m-1))$ when $2m + 1 $ is divisible by $3$, and this is the only case when the SNF of the matrices $D^{\deg}_+$ and $D^{\deg}$ of the star graph has only one invariant factor equal to one.
Let us analyze second distance ideal of the complete bipartite graph. 
For this, let 
\[
    D_{X,Y}(K_{m,n})=
    \begin{bmatrix}
        \diag(x_1,\dots,x_m)-2{\sf I}_m+2{\sf J}_m & {\sf J}_{m,n}\\
        {\sf J}_{n,m} & \diag(y_1,\dots,y_n)-2{\sf I}_n+2{\sf J}_n\\
    \end{bmatrix}.
\]
Since we are interested in the cases different from the star graphs, we have that if $m\geq2$ and $n\geq2$, then the second distance ideal of the complete is equal to
\[
\langle x_1-2, \dots, x_m-2, y_1-2, \dots, y_n-2,3\rangle.
\]
From which follows that the second invariant factor of $D^{\deg}_+(K_{m,n})$ is equal to 3 when $m-2$ and $n-2$ are divisible by 3, and the the second invariant factor of $D^{\deg}(K_{m,n})$ is equal to 3 when $m+2$ and $n+2$ are divisible by 3.
Since the only case when the second invariant factor of $D^{\deg}(K_n)$ is 3 is when $n=3$, and the second invariant factor of of $D^{\deg}$ of any complete bipartite with 3 vertices is different from 3, then we have the following result.

\begin{theorem}
    Complete graphs are determined by the SNF of $D^{\deg}$.
\end{theorem}

On the other hand, the only case when the second invariant factor of $D^{\deg}_+(K_n)$ is 3 is when $n=5$, and the second invariant factor of $D^{\deg}_+$ of any complete bipartite with 5 vertices is different from 3, then we have the following result.

\begin{theorem}
    Complete graphs are determined by the SNF of $D^{\deg}_+$.
\end{theorem}

\section{The Spectrum of $A^{\tr}$ and $D^{\deg}$}
\label{sec:spectrum}

For any real symmetric $n\times n$ matrix $M$, we write  
\[
\lambda_1(M) \leq \lambda_2(M) \leq \cdots  \leq \lambda_n(M)
\]
to denote the $n$ real eigenvalues of $M$. 
The Courant-Fisher theorem gives variational characterizations of the eigenvalues of $M$ as
\[
\lambda_i(M) \ = \min_{\dim(U)=i}\ \max_{\x\in U : \|\x\| = 1} \x^T M \x
\quad = \max_{\dim(U)=n-i-1}\ \min_{\x\in U:\|\x\| = 1} \x^T M \x
\]
for $i=1,\dots,n$, where $U$ ranges over all subspaces of $\mathbb{R}^n$.  
In particular, the smallest eigenvalue $\lambda_1(M)$ is given by 
\[
\lambda_1(M) \ = \  \min_{\|\x\| = 1} \x^T M \x
\]
and the largest $\lambda_n(M)$ by
\[
\lambda_n(M) \ = \ \max_{\|\x\| = 1} \x^T M \x.
\]

Let $\de(G)$ and $\De(G)$ denote the smallest and largest degrees of the vertices of a graph $G$, respectively. 
Similarly, let $\th(G)$ and $\Th(G)$ respectively denote the smallest and largest transmissions of the vertices of $G$. 
Recall that $A^{\tr} = \tr(G) - A$ and $D^{\deg} = \deg(G) - D$, where $\tr(G)$ and $\deg(G)$ are the diagonal matrices of transmissions and degrees, respectively.

\begin{theorem}\label{lem:bounds}
    If $G$ is a connected graph with $n$ vertices, then  
    \begin{align*}
            \la_1(A^{\tr}) \ &\geq \ \th(G) - \la_n(A),\quad & \la_1(D^{\deg}) \ &\geq \ \de(G) - \la_n(D),\\
            \la_n(A^{\tr}) \ &\leq \ \Th(G) - \la_1(A),\quad &
            \la_n(D^{\deg}) \ &\leq \ \De(G) - \la_1(D).
    \end{align*}
\end{theorem}
\begin{proof}[\emph{Proof}]
Let $M$ and $N$ be two real symmetric $n\times n$ matrices. 
If $\x \neq 0$ is a unit vector in $\re^{n}$, then 
\[
\x^T (M - N) \x \ =\  \x^T M \x - \x^T N \x
\]
and so
\begin{align*}
\la_1(M-N)\ = \ \min_{\|\x\| = 1 } \x^T (M - N) \x \ &\geq \ \min_{\|\x\| = 1 } \x^T M \x  -  \max_{\|\x\| = 1 } \x^T N \x \ = \ \la_1(M) - \la_n(N),\\
\la_n(M-N)\ = \ \max_{\|\x\| = 1 } \x^T (M - N) \x \ &\leq \ \max_{\|\x\| = 1 } \x^T M \x  -  \min_{\|\x\| = 1 } \x^T N \x \ = \ \la_n(M) - \la_1(N).
\end{align*}
Since $\la_1(\tr(G)) = \th(G)$, $\la_n(\tr(G)) = \Th(G)$, $\la_1(\deg(G)) = \de(G)$  and $\la_n(\deg(G)) = \De(G)$, the result follows..
\end{proof}

We have the following sufficient condition for the spectrum of the degree-distance matrix and the spectrum of the distance matrix to be equal up to a constant. 

\begin{lemma}\label{thm:sameDegreeSpectra}
Let $G$ be a graph with $n$ vertices such that every vertex of $G$ has the same degree $k$. Then $\lambda_i(D^{\deg}) =  k - \lambda_i(D)$ for all $i = 1,\dots,n$.
\end{lemma}
\begin{proof}[\emph{Proof}]
If $\x$ is an eigenvector of $M$ with eigenvalue $\lambda$, then  $\x$ is an eigenvector of $cI - M$ with eigenvalue $c - \lambda$ for any $c \in \re$. 
So if all the vertices of $G$ have the same degree $k$, we have $\deg(G) = k I$ and hence $D^{\deg}(G) = kI - D(G)$. 
Thus $\lambda_i(D^{\deg}) =  k - \lambda_i(D)$.
\end{proof}

Similarly, we have the following sufficient condition for the spectrum of the transmission-adjacency matrix and the spectrum of the adjacency matrix to be equal up to a constant.

\begin{lemma}\label{thm:sameTransmissionSpectra}
Let $G$ be a graph with $n$ vertices such that every vertex of $G$ has the same transmission $r$. 
Then  $\lambda_i(A^{\tr}) =  r - \lambda_i(A)$ for all $i = 1,\dots,n$.
\end{lemma}
\begin{proof}[\emph{Proof}]
If all the vertices of $G$ have the same transmission $r$, then $\tr(G) = r I$ and so $A^{\tr}(G) = r I - A(G)$. Hence,  an eigenvector of $A(G)$ with eigenvalue $\lambda_i(A)$ is an eigenvector of $A^{\tr}(G)$ with eigenvalue $r - \lambda_i(A)$.
\end{proof}

We now define a class of graphs for which the spectrum of $A^{\tr}$ is determined by the spectrum of the adjacency matrix, and the spectrum of $D^{\deg}$ is determined by the spectrum of the distance matrix. 
A graph $G$ of diameter $d$ is \emph{distance-regular} with intersection array $\{b_0,b_1,\dots, b_{d-1};c_1,c_2,\dots,c_d\}$ if $G$ is regular of degree $k = b_0$, and if for any two vertices $u$ and $v$ at distance $i$, there are precisely $c_i$ neighbors of $v$ at distance $i - 1$ from $u$, and $b_i$ neighbors of $v$ at distance $i+1$ from $u$. 
It is well known that two distance-regular graphs are cospectral with respect to the adjacency matrix if and only if they have the same intersection array (see e.g. \cite{brouwer1989}). 
Since the only vertex at distance 0 from a given vertex is itself,  we have that $c_1 = 1$.
The five Platonic solids are simple examples of distance-regular graphs: the tetrahedron has intersection array $\{3;1\}$, the octahedron $\{4,1; 1,4\}$, the cube $\{3,2,1; 1,2,3\}$, the icosahedron $\{5,2,1; 1,2,5\}$ and the dodecahedron $\{3,2,1,1,1; 1,1,1,2,3\}$.
Also, every strongly regular graph with parameters $(n,k,a,c)$ is a distance-regular graph with diameter at most 2 and  intersection array $\{k, k-a-1; 1, c\}$. 

The spectrum of the distance matrix of distance-regular graphs has been thoroughly studied, see for instance  \cite{drg-dsp, dsp-drg, dsp}. 
Since distance-regular graphs are \emph{transmission-regular} (i.e., each vertex has the same transmission), we have the following consequence of Lemma~\ref{thm:sameDegreeSpectra} and Lemma~\ref{thm:sameTransmissionSpectra}. 

\begin{corollary}
For distance-regular graphs, the spectrum of $D^{\deg}$ is determined by the spectrum of $D$ and the spectrum of $A^{\tr}$ is determined by the spectrum of $A$. \qed
\end{corollary}

Now we turn to study the relationship between the spectrum of $A^{\tr}$ and the combinatorial structure of the graph.  
Let $G$ be a graph with $n$ vertices. 
The \emph{Wiener index} $W(G)$ of $G$ is sum of the transmissions of all the vertices of $G$:
\[
W(G) \ = \sum_{u\in V(G)} \tr(u). 
\] 
We use $W^{\deg}(G)$ to denote the sum of the transmissions weighted by the degrees:
\[
W^{\deg}(G) \ = \sum_{u\in V(G)} \deg(u)\tr(u). 
\]

\begin{theorem}
Let $G$ be a connected graph with $n$ vertices and let $T(G)$ be the set of all triangles contained in $G$. Then
\begin{align*}
\sum_{i=1}^n \lambda_i(A^{\tr}) \ &=\ W(G),   \\
\sum_{i=1}^n \lambda_i(A^{\tr})^2  &=\ 2|E(G)|\ + \sum_{u\in V(G)} \tr(u)^2, \\
\sum_{i=1}^n \lambda_i(A^{\tr})^3  &=\ 6|T(G)|\ +\ W^{\deg}(G)\ +\sum_{u\in V(G)} \tr(u)^3.
\end{align*}
\end{theorem}
\begin{proof}[\emph{Proof}]
On the one hand, the sum of the diagonal entries of a matrix $M$ equals the sum of its eigenvalues.  
On the other, the eigenvalues of $M^k$  are the $k$-th powers of the eigenvalues of $M$. 
Since $\sum_{i=1}^n \lambda_i(A) = 0$, $\sum_{i=1}^n \lambda_i(A)^2 = 2|E(G)|$ and $\sum_{i=1}^n \lambda_i(A)^3 = 6|T(G)|$, the result follows. 
\end{proof}

Note that obtaining a similar expression for the fourth moment of the eigenvalues of $A^{\tr}$ seems to be complicated, because the fourth powers of the eigenvalues of the adjacency matrix of $G$ does not determine the number of 4-cycles. 

We write $u\sim v$ to denote that the vertices $u$ and $v$ are adjacent. 
Recall that the transmission $\tr(u)$ of a vertex $u$ from a graph $G$ is defined as 
\[
\tr(u)\ = \sum_{v\in V(G)} \dist(u,v). 
\]
Since the distance between any two adjacent vertices is 1, we have 
\[
\deg(u) \ = \sum_{v:u\sim v} \dist(u,v)
\]
where the sum is taken over all the vertices which are  adjacent to $u$.  
Hence we can always decompose the transmission of $u$ as 
\[
\tr(u) = 
\deg(u) + \sum_{v:u\not\sim v} \dist(u,v).
\]
If we let $R(G)$ be the diagonal matrix with $\tr(u) - \deg(u)$ in the diagonal, then
\[
A^{\tr}(G) = L(G) + R(G).
\]
Since $A^{\tr}$, $L$ and $R $ are real symmetric matrices, we can use Weyl's inequalities to obtain
\[
\max_{1\leq j\leq i} \{\lambda_{i-j+1}(L) + \lambda_{j}(R)\} \ \leq \  \lambda_i(A^{\tr}) \ \leq \min_{0\leq j\leq n-i} \{\lambda_{i+j}(L) + \lambda_{n-j}(R)\}
\]
for all $i = 1,\dots,n$. 
In particular, for each $i$ we have 
\[
\lambda_{i}(L) + \lambda_{1}(R) \ \leq \  \lambda_i(A^{\tr}) \ \leq \ \lambda_{i}(L) + \lambda_{n}(R).
\]

\begin{theorem}
Let $G$ be a connected graph with $n$ vertices. 
Then 
\[
\min_{u\in V(G)}\{\tr(u) - \deg(u)\}
\ \leq \
\lambda_1(A^{\tr}) 
\ \leq \ 
\frac{1}{n}\sum_{u\in V(G)} (\tr(u)-\deg(u)).
\]
\end{theorem}
\begin{proof}[\emph{Proof}]
For each vector $\x = (x_1,\dots, x_n)^T\in \re^n$ we have 
\[
\x^T A^{\tr} \x\ = \ \sum_{i\sim j}(x_i-x_j)^2 + \sum_i \lambda_i(R) x_i^2
\]
where the first sum is taken over all the unordered pairs $\{i,j\}$ that are edges of $G$. 
Let $\1 \in \re^n$ denote the vector with each entry equal to $1$.
Then 
\[
\lambda_1(A^{\tr})\ = \ \min_{\| \x \| = 1} \x^T A^{\tr} \x \ \leq\ \tfrac{1}{\sqrt{n}}\1^T A^{\tr} \tfrac{1}{\sqrt{n}}\1
\]
and so 
\[
 \lambda_1(A^{\tr})\ \leq\ \frac{1}{n}\sum_i \lambda_i(R). 
\] 
Since $R$ is the diagonal matrix that contains the differences $\tr(i) - \deg(i)$ in the diagonal, the average of its eigenvalues is equal to the average of these differences. 
This proves the upper bound. 
 For the lower bound, note that $\lambda_1(L) = 0$ always. Then 
 \[
 \lambda_{1}(R) \ \leq \  \lambda_1(A^{\tr}). 
 \]
 The result follows from the fact that the smallest eigenvalue $\lambda_{1}(R)$ of $R$ is equal to the smallest entry of $R$. 
\end{proof}

A \emph{generalized Laplacian} of $G$ is any symmetric $n\times n$ matrix $M$ satisfying $M_{uv} < 0$ if   $u\sim v$ and $M_{uv} = 0$ if $u\neq v$ and $u\not\sim v$. 
For instance, the Laplacian matrix $L$ is a generalized Laplacian, and so is the transmission-adjacency matrix $A^{\tr}$. 

\begin{theorem}{\cite{agt}}
    If $G$ is connected, then $\lambda_1(A^{\tr})$ is simple and the corresponding eigenvector can be taken to have all its entries positive.
\end{theorem}

The smallest non-trivial eigenvalue of $L$ is $\lambda_2(L)$ because $\lambda_1(L) = 0$ always. 
This eigenvalue can be computed as
\[
\lambda_2(L) \ =  \min_{\|\x\| = 1 : \x^T\1 = 0} \x^T L \x
\]
where the minimum is taken over all unit vectors which are orthogonal to the all-ones vector $\1$, see e.g.~\cite[Corollary 13.4.2]{agt}. 
Now for any subset of vertices $S \subseteq V(G)$, the \emph{edge-boundary} $\partial S$ of $S$ is defined as $\partial S = \{\{u,v\}\in E(G) : u\in S, v\in V(G)\setminus S\}$. 
The \emph{conductance} $\Phi(G)$ of a graph $G$ on $n$ vertices is defined as 
\[
\Phi(G) \ = \min_{S \subseteq V(G):|S| \leq n/2} \frac{|\partial S|}{|S|}. 
\]
It is a well-known result (see~\cite[Lemma 13.7.1]{agt}) that 
\[
\lambda_2(L)\ \leq\ \frac{n|\partial S|}{|S|(n - |S|)}
\]
from which it follows that
\[
\lambda_2(L)\ \leq\ 2 \Phi(G).
\]
Also, it can be shown that
\[
\frac{\Phi(G)^2}{2\Delta(G)}\ <\ \lambda_2(L).
\]

\begin{theorem}\label{thm:boundscomb}
Let $G$ be a connected graph with $n$ vertices. 
Then
\[
\frac{\Phi(G)^2}{2\Delta(G)} + \min_{u\in V(G)}\{\tr(u) - \deg(u)\}\ <\ \lambda_2(A^{\tr})\ \leq\ 2 \Phi(G) + \max_{u\in V(G)}\{\tr(u) - \deg(u)\}.
\]
\end{theorem}
\begin{proof}[\emph{Proof}] 
By Weyl’s inequalities, we have 
\[
\lambda_2(L) + \lambda_1(R)\ \leq \ \lambda_2(A^{\tr})\ \leq\ \lambda_2(L) + \lambda_n(R).
\]
Since 
$\lambda_1(R) = \min\{\tr(u) - \deg(u) : u \in V(G)\}$ and $\lambda_n(R) = \max\{\tr(u) - \deg(u) : u \in V(G)\}$, the result follows from the above discussion concerning the conductance and maximum degree of $G$. 
\end{proof}

 Note that the bounds for $\lambda_2(A^{\tr})$ provided in Theorem~\ref{thm:boundscomb} are worst than the spectral bounds given by Weyl's inequalities.  
 We include them here just because they have a purely combinatorial interpretation.
 
Finally, we have the following results which relate the multiplicity of $\lambda_2(A^{\tr})$ to certain structural properties of the graph. 
These two results are true for any generalized Laplacian.

\begin{theorem}{\cite{agt}}
    If $G$ is 2-connected and outerplanar, then $\lambda_2(A^{\tr})$ has multiplicity at most two. 
\end{theorem}

\begin{theorem}{\cite{agt}}
    If $G$ is 3-connected and planar, then $\lambda_2(A^{\tr})$ has multiplicity at most three. 
\end{theorem}

\section*{Acknowledgement}
The research of C.A. Alfaro is partially supported by the Sistema Nacional de Investigadores grant number 220797. 
The research of O. Zapata is partially supported by the Sistema Nacional de Investigadores grant number 620178.

\end{document}